\documentclass[10pt,a4paper]{article}
\usepackage{fullpage}
\usepackage{amsfonts,amsmath,amssymb}
\usepackage{amsthm}
\usepackage{graphicx}\usepackage{sectsty}
\usepackage{authblk}

\theoremstyle{plain}
\newtheorem{theorem}{Theorem}[section]
\newtheorem{lemma}[theorem]{Lemma}

\newtheorem{corollary}[theorem]{Corollary}

\theoremstyle{definition}
\newtheorem{definition}[theorem]{Definition}
\theoremstyle{remark}

\numberwithin{equation}{section}

\sectionfont{\large}
\newenvironment{data availability}[1][Data Availability
]{\begin{trivlist} \item[\hskip \labelsep {\bfseries
#1}]}{\end{trivlist}}
\makeatletter
\newcommand{\opnorm}{\@ifstar\@opnorms\@opnorm}
\newcommand{\@opnorms}[1]{%
  \left|\mkern-1.5mu\left|\mkern-1.5mu\left|
   #1
  \right|\mkern-1.5mu\right|\mkern-1.5mu\right|
}
\newcommand{\@opnorm}[2][]{%
  \mathopen{#1|\mkern-1.5mu#1|\mkern-1.5mu#1|}
  #2
  \mathclose{#1|\mkern-1.5mu#1|\mkern-1.5mu#1|}
}
\makeatother
\begin{document}
\title{Certain Inverse Resonance Uniqueness on the Line with Super-Exponentially Decaying Potential}
\author{Lung-Hui Chen$^1$}\maketitle\footnotetext[1]{General Education Center, Ming Chi University of Technology, New Taipei City, 24301, Taiwan. Email:
mr.lunghuichen@gmail.com.}\maketitle
\begin{abstract}
In the paper, we study the inverse problem with the resonant data of fast decaying potential $V$. We review Froese' construction of the Born's approximation and Neumann series to analyze the growth of scattering determinant. Assuming all the the resonances are given, we deduce the certain inverse uniqueness on $V$ from the Nevanlinna type of representation theorem.
\end{abstract}
\section{Introduction}
Scattering phenomenon happens in many branches of sciences and engineering.
A physical state typically has a rate of oscillation and a rate of decay. Scattering resonances generalize eigenvalues or bound states for those physical states in
which energy can scatter to infinity with a rate of oscillation and in a rate of decaying. The concept appears to be intrinsically dynamical, 
but the mathematical formulation comes from
considering the meromorphic continuations of Green's functions or scattering
matrices. The poles of these meromorphic continuations capture the physical information by identifying the rate of oscillations with the real part of
a pole, and the rate of decay with its imaginary part. 
\par
Mathematically, resonances are defined as the poles in the meromorphic extension of resolvent 
\begin{equation}\label{111}
R(k):=(-D^{2}+V-k^{2})^{{-1}}:L^{2}(\mathbb{R})\rightarrow L^{2}(\mathbb{R})
\end{equation}
from the upper half complex plane $\mathbb{C}^{+}:=\{\Im k>0\}$ to the whole complex plane, where $D$ is the differentiation to space varaiable. In the lower half complex $\mathbb{C}^{-}$, the continuation will not exist as an operator in $L^{2}(\mathbb{R})$. Therefore, we consider the extended resolvent as maps between suitable spaces of exponentially weighted spaces \cite{Dyatlov,Froese,Hitrik,Zworski}. In particular, we use the following definition.
\begin{definition}
A resonance is a pole in the meromorphic continuation of $V^{\frac{1}{2}}R(k)|V|^{\frac{1}{2}}$.
\end{definition}
Here, $V(x)^{\frac{1}{2}}$ is the short for
\begin{eqnarray*}
V(x)^{\frac{1}{2}}=\left\{%
\begin{array}{ll}
V(x)/|V(x)|^{\frac{1}{2}}&,\mbox{ if }V(x)\neq0;\\
0&,\mbox{ if }V(x)=0.
\end{array}%
\right.
\end{eqnarray*}
\begin{definition}
We say the potential $V(x)$ is super-exponentially-decaying if for every $N\in\mathbb{R}$, there is a constant $C_{N}$ such that
\begin{equation}
|V(x)|\leq C_{N}e^{-N|x|}.
\end{equation}
\end{definition}
In this paper, we consider the super-exponentially-decaying potentials to ensure that $V^{\frac{1}{2}}R(k)|V|^{\frac{1}{2}}$ has a meromorphic extension  \cite{Barry,Dyatlov,Froese,Froese2,Hitrik,Zworski}.
Now we let 
$$R_{0}(k):=(-D^{2}-k^{2})^{{-1}}$$to denote the free resolvent.
We begin with the resolvent formula 
$$R(k)=R_{0}(k)-R_{0}(k)VR(k)$$
In this case, we deduce
$$(1+V^{\frac{1}{2}}R_{0}(k)|V|^{\frac{1}{2}})V^{\frac{1}{2}}R(k)|V|^{\frac{1}{2}}=V^{\frac{1}{2}}R_{0}(k)|V|^{\frac{1}{2}}.$$
Let us set $$\mathbf{R}_{V}(k):=V^{\frac{1}{2}}R_{0}(k)|V|^{\frac{1}{2}}:L^{2}(\mathbb{R})\rightarrow L^{2}(\mathbb{R}),$$
and we deduce 
$$(1+\mathbf{R}_{V}(k))V^{\frac{1}{2}}R(k)|V|^{\frac{1}{2}}=\mathbf{R}_{V}(k).$$
\begin{lemma}
If $V$ is super-exponentially-decaying, then the operator $\mathbf{R}_{V}(k)$  is in trace class for $\{\Im k>0\}$, and has a trace class operator-valued analytic extension to $\mathbb{C}$.
\end{lemma}
\begin{proof}
We refer the proof to Froese \cite[Lemma\,3.1]{Froese}.
\end{proof}
\begin{definition}
Let us define $$D(k):=\det(1+\mathbf{R}_{V}(k)).$$
\end{definition}
Therefore, resonances are exactly those values of $k\in\mathbb{C}$ such that $(1+\mathbf{R}_{V}(k))$ is not invertible. Hence, the resonances are the zeros of functional determinant $D(k)$. Moreover,
the scattering determinant $$E(-k)=\frac{D(k)}{D(-k)},$$
$$E(-k)=\det\big(  \left[\begin{array}{cc}1& 0 \vspace{5pt}\\0 & 1\end{array}\right]+\frac{i}{2k}\left[\begin{array}{cc}T_{11}& T_{12} \vspace{5pt}\\T_{21} & T_{22}\end{array}\right] \big),$$
and
\begin{eqnarray*}
&&T_{11}(-k)=\int V(x)(1-f_{+}(x,k))dx;\\
&&T_{12}(-k)=\int e^{2ikx}V(x)(1-f_{-}(x,k))dx;\\
&&T_{21}(-k)=\int e^{-2ikx}V(x)(1-f_{+}(x,k))dx;\\
&&T_{22}(-k)=\int V(x)(1-f_{-}(x,k))dx,
\end{eqnarray*}
in which 
\begin{eqnarray*}
&&f_{+}(x,k):=e^{ikx}R(-k)e^{-ikx}V;\\
&&f_{-}(x,k):=e^{-ikx}R(-k)e^{ikx}V.
\end{eqnarray*}
If $V\in L^{1}(\mathbb{R})$, then \cite[p.\,258]{Froese}
\begin{equation}\label{3.9}
|f_{\pm}(x,k)|\leq C/k,\,\Im k<0.
\end{equation}
\par
We divide the complex plane into four sectors: Let $\alpha=\pi/\rho$, $\rho>1$, and
\begin{eqnarray}
&&\Gamma=\{k:\,|\arg k|\leq\frac{\pi-\alpha}{2}\};\\
&&-\Gamma=\{k:-k\in\Gamma\};\\
&&\Gamma_{-}=\{k:\,|\arg k+\frac{\pi}{2}|\leq\frac{\alpha}{2}\};\\
&&-\Gamma_{-}=\{k:-k\in\Gamma_{-}\}.
\end{eqnarray}
Let us state the result in this paper under the following hypotheses:
\begin{description}
\item[H1]$\hat{V}(k)$ has order $\rho>1$, is of finite type $\sigma$, and is of completely regular growth.
\item[H2]There exists a positive $b$ such that the Fourier transform $\hat{V}(k)$ satisfying
\begin{equation}\label{H1}
|\hat{V}(k)|+|\hat{V}'(k)|+|\hat{V}''(k)|\leq e^{b|\Im k|},\end{equation} for $k\in\pm\Gamma$.
\item[H3]Let $C$ denote the constant in~(\ref{3.9}). There exists $\delta>0$ such that for a set of real $\lambda$ of density one
$$\widehat{|V|}(2i\lambda)\leq\frac{1-\delta}{C}|\lambda||\hat{V}(2i\lambda)|.$$ 
\item[H4]The zero set of $\hat{V}(k)$ are on  either $\mathbb{C}^{+}$ or $\mathbb{C}^{-}$.
\end{description} 
The hypothesis {\bf H1}, {\bf H2}, and {\bf H3} are Froese' Hypotheses 5.1 in \cite{Froese}. Here, we add {\bf H4} to control the distribution of zeros of $\hat{V}(k)$. In particular, that is a special class of HB functions, which we refer to \cite{Levin,Levin2}. The zero set of $\hat{V}(k)$     plays a roll in this paper as the approximations of the scattering resonances.
\par
In literature, when there are no bound states, the potential is determined by the reflection coefficients by Faddeev's theory \cite{Ak98,CS89,DT79,Fa64}. 
In inverse resonance problem, we consider to determine the potential $V$ from the resonances of~(\ref{111}) which includes the square root of $L^{2}$-eigenvalues. Such an inverse process to identify the knowledge of  the emitting source by measuring all sorts of respects of the emittance or perturbed wavefield in observational area has been research issues ever since the days of A. Sommerfeld and E. Rutherford. The inverse resonance problem of Schr\"{o}dinger operator on the half line has been studied in \cite{Korotyaev1,Korotyaev2,Korotyaev3,Marletta,Xu,Zworski2}. In the half line case, the unique recovery of the potential from the eigenvalues and resonances is justified in \cite{Korotyaev1}. Moreover, if the potential is known a priori on a larger interval, then infinitely many resonances can be removed from the unique determination of the potential on the interval \cite{Xu}. However, in the full line case, the inverse resonance problems mainly remain open for a long time. It is known that the potential cannot be solely determined by the eigenvalues and resonances. Specifically, Zworski \cite{Zworski2} proved the uniqueness theorem for the symmetric potentials along with certain isopolar results. Furthermore, Korotyaev \cite{Korotyaev1,Korotyaev3} applied the value distribution theory in complex analysis to prove that all eigenvalues and resonances, and a signed sequence can uniquely determine the potential $V$. Moreover, Bledsoe \cite{Beldsoe} and Korotyaev \cite{Korotyaev2} studied the stability of inverse resonance problem. 
\par
We state the result of this paper.
\begin{theorem}\label{11}
If we have two potential function $V^{1}$ and $V^{2}$ assumed as above, then we deduce
\begin{equation}\nonumber
V^{1}(x)=\pm V^{2}(x).
\end{equation}
\end{theorem}
\section{Complex Analysis and Froese' Theorem}
The Fourier transform $\hat{V}(z)$ behaves like the exponential functions in some sectors in $\mathbb{C}$. The Nevanlinna type of  representation theorem plays a role.
\begin{definition}\label{71}
Let $f(z)$ be an entire function. Let
$M_f(r):=\max_{|z|=r}|f(z)|$. An entire function of $f(z)$ is said
to be a function of finite order if there exists a positive
constant $k$ such that the inequality
\begin{equation}\nonumber
M_f(r)<e^{r^k}
\end{equation}
is valid for all sufficiently large values of $r$. The greatest
lower bound of such numbers $k$ is called the {\bf order} of the entire
function $f(z)$. By the {\bf type } $\sigma$ of an entire function $f(z)$
of order $\rho$, we mean the greatest lower bound of positive
number $A$ for which asymptotically we have
\begin{equation}\nonumber
M_f(r)<e^{Ar^\rho}.
\end{equation}
That is,
\begin{equation}\nonumber
\sigma:=\limsup_{r\rightarrow\infty}\frac{\ln M_f(r)}{r^\rho}.
\end{equation}  If $0<\sigma<\infty$, then we say
$f(z)$ is of normal type or mean type.
\end{definition}
\begin{definition}\label{72}
Let $f(z)$ be an integral function of finite order $\rho$ in the
angle $[\theta_1,\theta_2]$. We call the following quantity as the
indicator function of the function $f(z)$.
\begin{equation}\nonumber
h_f(\theta):=\lim_{r\rightarrow\infty}\frac{\ln|f(re^{i\theta})|}{r^{\rho}},
\,\theta_1\leq\theta\leq\theta_2.
\end{equation}
\end{definition}
The type of a function is connected to the maximal value of indicator function.
\begin{lemma}[Levin \cite{Levin},\,p.72]\label{L4}
The maximal value of indicator function $h_F(\theta)$ of
$F(z)$ on the interval $\alpha\leq\theta\leq\beta$ is equal to the
type $\sigma_F$ of this function inside the angle $\alpha\leq\arg
z\leq\beta$.
\end{lemma}
\begin{lemma}\label{266}
Let $f$, $g$ be two entire functions. Then the following two
inequalities hold.
\begin{eqnarray}
&&h_{fg}(\theta)\leq h_{f}(\theta)+h_g(\theta),\mbox{ if one limit exists};\label{119}\\\label{120}
&&h_{f+g}(\theta)\leq\max_\theta\{h_f(\theta),h_g(\theta)\},
\end{eqnarray}
where the equality in~(\ref{119}) holds if one of the functions is of completely regular growth, and secondly the equality~(\ref{120}) holds if the indicator of the two summands are not equal at some $\theta_0$.
\end{lemma}
Using Froese's theory \cite{Froese}, we state the following lemma.
\begin{lemma}
The scattering determinant
$$E(-k)=\frac{D(k)}{D(-k)}=\det\big(  \left[\begin{array}{cc}1& 0 \vspace{5pt}\\0 & 1\end{array}\right]+\frac{i}{2k}\left[\begin{array}{cc}T_{11}& T_{12} \vspace{5pt}\\T_{21} & T_{22}\end{array}\right] \big);$$
For $k\in\mathbb{C}^{-}$, we have
\begin{eqnarray}\label{2.1}
&&T_{11}(k)=\hat{V}(0)+O(\frac{1}{|k|}),\\\label{2.2}
&&T_{22}(k)=\hat{V}(0)+O(\frac{1}{|k|}),\\
&&D(k)=1+\frac{1}{4k^{2}}T_{12}(-k)T_{21}(-k)+O(\frac{1}{|k|}).\label{2.3}
\end{eqnarray}
\end{lemma}
\begin{proof}
We refer the proof of \cite[Corollary\,3.4]{Froese}
\end{proof}
\begin{lemma}
We have $$\lim_{|k|\rightarrow\infty}D(k)=1,$$ for any ray in $\mathbb{C}^{+}$. For $k$ real, $D(k)$ is bounded for large $k$.
\end{lemma}
\begin{proof}
We refer the proof of \cite[Lemma\,3.2]{Froese}.
\end{proof}
Hence, we deduce the following lemma.
\begin{lemma}
$D(k)$ has only finitely many zeros in $\mathbb{C}^{+}$.
\end{lemma}
\begin{proof}
Using Phragm\'{e}n-Lindel\"{o}f theorem  \cite[p\,38]{Levin2} on upper complex plane $\mathbb{C}^{+}$, we have $D(k)\rightarrow1$ uniformly along any ray in $\mathbb{C}^{+}$.
This proves the lemma.
\end{proof}
\begin{theorem}[Froese]
As the assumption of Theorem \ref{11}, the indicator function of $D(k)$ is given by
\begin{eqnarray}
h_{D}(\theta)=\left\{%
\begin{array}{ll}
2^{\rho}(h_{\hat{V}}(\frac{\pi}{2})+h_{\hat{V}}(-\frac{\pi}{2}))\cos(\rho(\theta+\frac{\pi}{2}))&,|\theta+\frac{\pi}{2}|<\frac{\alpha}{2};\vspace{5pt}\\
0&,\mbox{otherwise}.
\end{array}%
\right.
\end{eqnarray}
\end{theorem}
\begin{proof}
We refer the proof to \cite[Theorem\,1.3]{Froese}
\end{proof}
\begin{corollary}
The indicator function $h_{D}(\theta)$ is non-zero if $V$ is non-trivial.
\end{corollary}
\begin{proof}
$D(k)$ is bounded on the real axis. If $h_{D}(\theta)\equiv0$, then it is of zero type in $\mathbb{C}$, that is, in $\pm\Gamma$ and $\pm\Gamma_{-}$. Using Phragm\'{e}n-Lindel\"{o}f theorem  \cite[p\,38]{Levin2} on upper complex plane $\mathbb{C}^{+}$, we have $D(k)$ is bounded on $\mathbb{C}^{+}$.
\par
Moreover, $D(k)$ is bounded on positive real axis and at most exponential growth on the lower boundary of $\Gamma$. Using Phragm\'{e}n-Lindel\"{o}f theorem again, we obtain the $D(k)$ is bounded on $\Gamma$. Similarly, $D(k)$ is bounded on $-\Gamma$. 
\par
In $$\Gamma_{-}=\{k:\,|\arg k+\frac{\pi}{2}|\leq\frac{\pi}{2\rho}\},$$
we use Phragm\'{e}n-Lindel\"{o}f theorem \cite[p\,38]{Levin2} once more on $\Gamma_{-}$. It is bounded on the boundary of the $\Gamma_{-}$, we deduce that $D(k)$ is uniformly bounded in $\Gamma_{-}$. Then, $D(K)$ is a constant in $\mathbb{C}$ by Liouville's theorem. The constant is $1$ by \cite[Lemma\,3.2]{Froese}. 
Using~(\ref{2.3}), we see that $$T_{12}(-k)T_{21}(-k)\equiv0.$$
That is, $$\hat{V}(2k)\hat{V}(-2k)=|\hat{V}(2k)|^{2}=0,$$
for real $k$. Hence, $\hat{V}(2k)\equiv0$, and then, $V$ is trivial. This is contradiction.
\end{proof}
\begin{corollary}\label{211}
Most of the zeros of $\hat{V}$ are in $\mathbb{C}^{-}$ in the sense that the number of resonances of modulus less then $r$  is given by
\begin{equation}n(r)=\frac{2^{\rho}(h_{\hat{V}}(\frac{\pi}{2})+h_{\hat{V}}(-\frac{\pi}{2}))}{2\pi}r^{\rho}+o(r^{\rho}),\label{277}
\end{equation}
In any other sector outside $\Gamma_{-}$, the number of resonances of modulus less than $r$ is $o(r^{\rho})$. Here,
\begin{eqnarray}\label{2.8}
h_{\hat{V}}(\theta)=a\cos(\rho(\theta+\frac{\pi}{2})),\,|\theta+\frac{\pi}{2}|<\frac{\alpha}{2},\mbox{ for some }a.
\end{eqnarray}
\end{corollary}
\begin{proof}
We refer the proof to \cite[Theorem\,1.3]{Froese}.
\end{proof}
\begin{theorem}[Govorov]\label{Govorov}
Let $f(z)$ be a function of completely regular growth of finite order $\rho\geq0$
in the half plane $\Im z>0$. Then, the following representation theorem of $F(z)$ holds.
\begin{eqnarray}\nonumber
f(z)&=&\exp[i(a_{0}+a_{1}z+\cdots+a_{q}z^{q})]\times\exp\frac{1}{\pi i}\int_{-\infty}^{\infty}\big(\frac{tz+1}{t^{2}+1}\big)^{q+1}\frac{d\sigma(t)}{t-z}\times \prod_{|z_{k}|\leq1}\frac{z-z_{k}}{z-\overline{z}_{k}}\\&&\times  \prod_{|z_{k}|>1}\frac{1-z/z_{k}}{1-z/\overline{z}_{k}}\exp 2i\big(z\Im\frac{1}{z_{k}}+\frac{z^{2}}{2}\Im\frac{1}{z_{k}^{2}}+\cdots+\frac{z^{q}}{q}\Im\frac{1}{z_{k}^{q}} \big),
\end{eqnarray}
where $q=[\rho]$, $\{a_{k}\}$ are real constants, $\{z_{k}\}$ are the zeros of $f(z)$ in $\Im z>0$, 
and
\begin{equation}
\sigma(t):=\lim_{y\rightarrow0^{+}}\int_{0}^{t}\ln|f(x+iy)|dx,
\end{equation}
which is of locally bounded variation.
\end{theorem}
\begin{proof}
We refer the discussion to \cite[p.\,472]{Levin} and Govorov \cite{Govorov}.
\end{proof}

\section{Proof of Theorem \ref{11}}
Let $V^{1}(x)$ and $V^{2}(x)$ be two potentials with the same set of resonances $\mathcal{R},$ with
$D^{1}(k)=\det(1+\mathbf{R}_{V^{1}}(k))$ and $D^{2}(k)=\det(1+\mathbf{R}_{V^{2}}(k))$ be the determinants respectively. 
\par
Let us consider $\frac{D^{1}(k)}{D^{2}(k)}$ which is bound on real $k$, since $D^{2}(k)$ has no zero on the real axis and $D^{1}(k)$ is bounded on the real $k$, which we refer to \cite[Lemma\,3.2]{Froese}. The zeros of the determinant $D^{j}(k)$ are the resonances of $V^{j}(x)$, $j=1,2.$ If they share the same resonant set $\mathcal{R}$, then 
$$\mathcal{D}(k):=\frac{D^{1}(k)}{D^{2}(k)}$$
is bounded on real axis and entire in $\mathbb{C}$. 
Now we use Definition \ref{72} to compute the indicator function of $\mathcal{D}(k)$, from which we deduce that
$$h_{\mathcal{D}}(\theta)=h_{D^{1}}(\theta)-h_{D^{2}}(\theta)=0,\,\theta\in[0,2\pi],$$
in which the first equality is due to Lemma \ref{266}, and the second equality is due to~(\ref{277}). Therefore, $\mathcal{D}$ is bounded on the real axis and of zero type in $\mathbb{C}$. We deduce form Phragm\'{e}n-Lindel\"{o}f theorem \cite{Levin} that  $\mathcal{D}$ is bounded on the upper and lower half  complex plane. Using Liouville's theorem, we deduce that $\mathcal{D}$ is a constant. Moreover,
$D(k)\rightarrow1$ along any ray in the upper half complex plane \cite[Lemma\,3.2]{Froese}, we deduce that $$\mathcal{D}(k)\equiv1,\,k\in\mathbb{C}.$$
Thus,
$$D^{1}(k)\equiv D^{2}(k),\,k\in\mathbb{C}.$$
Combining with~(\ref{2.1}),~(\ref{2.2}), and~(\ref{2.3}), we obtain that
$$0\equiv D^{1}(k)-D^{2}(k)=\frac{1}{4k^{2}}T_{12}^{1}(-k)T_{21}^{1}(-k)-\frac{1}{4k^{2}}T_{12}^{2}(-k)T_{21}^{2}(-k)+O(\frac{1}{|k|}).$$
Considering taking the limit $k\downarrow\{0-i\infty\}$, we deduce that
\begin{equation}\label{3.1}
T_{12}^{1}(-k)T_{21}^{1}(-k)\equiv T_{12}^{2}(-k)T_{21}^{2}(-k),
\end{equation}
in which 
\begin{eqnarray}\label{3.2}
&&T^{j}_{21}(-k)=\hat{V}^{j}(2k)+\sum_{n=1}^{\infty}(-1)^{n}([A_{k}^{j}]^{n}\hat{V}^{j})(2k);\\\label{3.3}
&&T^{j}_{12}(-k)=\hat{V}^{j}(-2k)+\sum_{n=1}^{\infty}(-1)^{n}([A_{k}^{j}]^{n}\hat{V}^{j})(-2k),
\end{eqnarray}
in which we refer the identities to Froese \cite[p.\,262]{Froese}. Most important of all, we have
\begin{equation}\label{3.4}
\opnorm{A_{\pm k}^{j}\hat{V}}\leq (\frac{C\ln|\Im k|}{|\Im k|})\opnorm{\hat{V}},\,k\in\Gamma.
\end{equation}
Combining~(\ref{3.2}),~(\ref{3.3}), and~(\ref{3.4}), we deduce for non-real $k\in\Gamma$ that
\begin{equation}\label{3.5}
T_{12}^{j}(-k)T_{21}^{j}(-k)=\hat{V}^{j}(2k)\hat{V}^{j}(-2k)+O(\frac{\ln|\Im k|}{|\Im k|}).
\end{equation}
Using~(\ref{3.1}) and~(\ref{3.5}), we identify the higher order term in $\Gamma$ to deduce 
\begin{equation}\label{3.6}
\hat{V}^{1}(2k)\hat{V}^{1}(-2k)=\hat{V}^{2}(2k)\hat{V}^{2}(-2k).
\end{equation}
Then, using~(\ref{119}),
\begin{equation}
h_{\hat{V}^{1}}(\theta)+h_{\hat{V}^{1}}(-\theta)=h_{\hat{V}^{2}}(\theta)+h_{\hat{V}^{2}}(-\theta).
\end{equation}
Using~(\ref{2.8}), $\hat{V}^{1}(k)$ and $\hat{V}^{2}(k)$ are of the same type. 
For $V$ is real-valued, we deduce from~(\ref{3.6}) that
\begin{equation}\label{3.7}
|\hat{V}^{1}(k)|=|\hat{V}^{2}(k)|,\,k\in\mathbb{R}+0i.
\end{equation}
To use Theorem \ref{Govorov}, we compute the following functions of locally bounded variation:
\begin{equation}\label{399}
\sigma^{j}(t):=\lim_{y\rightarrow0^{+}}\int_{0}^{t}\ln|\hat{V}^{j}(x+iy)|dx,\,j=1,2.
\end{equation}
Now we deduce from~(\ref{3.7}) and~(\ref{399}) that
\begin{equation}
\sigma^{1}(t)=\sigma^{2}(t),
\end{equation}
and so that
\begin{eqnarray}\nonumber
&&\hat{V}^{1}(z)\exp[-i(a_{0}^{1}+a_{1}^{1}z+\cdots+a_{q}^{1}z^{q})]\times \prod_{|z_{k}^{1}|\leq1}\frac{z-\overline{z}_{k}^{1}}{z-z_{k}^{1}}\\\nonumber
&&\times  \prod_{|z_{k}^{1}|>1}\frac{1-z/\overline{z}_{k}^{1}}{1-z/z_{k}^{1}}\exp -2i\big(z\Im\frac{1}{z_{k}^{1}}+\cdots\big)\\\nonumber
&=&\hat{V}^{2}(z)\exp[-i(a_{0}^{2}+a_{1}^{2}z+\cdots+a_{q}^{2}z^{q})]\times \prod_{|z_{k}^{2}|\leq1}\frac{z-\overline{z}_{k}^{2}}{z-z_{k}^{2}}\\\label{pole}
&&\times  \prod_{|z_{k}^{2}|>1}\frac{1-z/\overline{z}_{k}^{2}}{1-z/z_{k}^{2}}\exp -2i\big(z\Im\frac{1}{z_{k}^{2}}+\cdots\big)
\end{eqnarray}
where $\{z_{k}^{j}\}$ are the zeros of the function $\hat{V}^{j}(k)$, $j=1,2$, in the half-plane $\Im z>0$. Using {\bf H4}, we compare and drop the zeros in $$\prod_{|z_{k}^{j}|\leq1}\frac{z-\overline{z}_{k}^{j}}{z-z_{k}^{j}}\times  \prod_{|z_{k}^{j}|>1}\frac{1-z/\overline{z}_{k}^{j}}{1-z/z_{k}^{j}}\exp -2i\big(z\Im\frac{1}{z_{k}^{j}}+\cdots\big)
,$$ $j=1,2$, on both sides of~(\ref{pole}). 
Therefore,
\begin{equation}\nonumber
\hat{V}^{1}(z)=e^{-iQ(z)}\hat{V}^{2}(z),
\end{equation}
where $Q(z)$ is a polynomial of order $q$ with real coefficients.
Taking logarithm on both sides, we deduce that
$$\ln\hat{V}^{1}(z)=-iQ(z)+\ln\hat{V}^{2}(z),$$
and then
$$\ln|\hat{V}^{1}(z)|+i\arg\hat{V}^{1}(z)=-iQ(z)+\ln|\hat{V}^{2}(z)|+i\arg\hat{V}^{2}(z),$$
Using~(\ref{3.7}) on real axis, we obtain
$$Q(\Re  z)=\arg\hat{V}^{2}(\Re z)-\arg\hat{V}^{1}(\Re z)=(a_{0}^{2}-a_{0}^{1})+(a_{1}^{2}-a_{1}^{1})\Re z+\cdots+(a_{q}^{2}-a_{q}^{1})(\Re z)^{q}.$$
However, $\arg\hat{V}^{1}(\Re z)-\arg\hat{V}^{1}(\Re z)$ is a bounded for for all $\Re z$, so we 
deduce $Q(\Re  z)=a_{0}^{2}-a_{0}^{1}$. Moreover, $V^{1}$ and $V^{2}$ are real-valued, so $a_{0}^{2}-a_{0}^{1}=0$ or $\pi$. 
The theorem is thus proven.
\begin{data availability}
Data sharing not applicable to this article as no datasets were generated or analysed during the current study.
\end{data availability}

\end{document}